\documentclass{amsart}
\usepackage{a4wide}

\usepackage{comment}
\usepackage[alphabetic, initials]{amsrefs}
\usepackage[T1]{fontenc}
\usepackage{amssymb}

\newtheorem{thm}{Theorem}[section]
\newtheorem{cor}[thm]{Corollary}
\newtheorem{lem}[thm]{Lemma}
\newtheorem{prop}[thm]{Proposition}

\newtheorem{thmx}{Theorem}

\theoremstyle{definition}
\newtheorem{defin}[thm]{Definition}

\newtheorem{rem}[thm]{Remark}

\newcommand{\lle}{\langle}
\newcommand{\rre}{\rangle}

\newcommand{\eps}{\varepsilon}
\newcommand{\red}{W_3}

\DeclareMathOperator{\Var}{Var}

\BibSpec{collection.article}{%
    +{}  {\PrintAuthors}                {author}
    +{,} { \textit}                     {title}
    +{.} { }                            {part}
    +{:} { \textit}                     {subtitle}
    +{,} { \PrintContributions}         {contribution}
    +{,} { \PrintConference}            {conference}
    +{}  {\PrintBook}                   {book}
    +{,} { }                            {booktitle}
    +{,} { }                            {publisher}
    +{,} { \PrintDateB}                 {date}
    +{,} { pp.~}                        {pages}
    +{,} { }                            {status}
    +{,} { \PrintDOI}                   {doi}
    +{,} { available at \eprint}        {eprint}
    +{}  { \parenthesize}               {language}
    +{}  { \PrintTranslation}           {translation}
    +{;} { \PrintReprint}               {reprint}
    +{.} { }                            {note}
    +{.} {}                             {transition}
    +{}  {\SentenceSpace \PrintReviews} {review}
}

\title[Non-orderability of random triangular groups]{Non-orderability of random triangular groups \\by using random 3CNF formulas}

\author{Damian Orlef}
\address{Institute of~Mathematics of the Polish Academy of~Sciences, ul. \'{S}niadeckich 8, 00-656 Warsaw, Poland}
\email{dorlef@impan.pl}

\begin{document}
\maketitle

\begin{abstract}
We show that a~random group $\Gamma$ in the triangular binomial model $\Gamma(n, p)$ is a.a.s.\ not left-orderable for $p\in(cn^{-2}, n^{-3/2-\eps})$, where $c, \eps$ are any constants satisfying $\eps>0$, ${c>(1/8)\log_{4/3}{2}\approx 0.3012}$. We also prove that if $p\geq (1+\eps)(\log n)n^{-2}$ for any fixed $\eps>0$, then a~random $\Gamma\in \Gamma(n,p)$ has a.a.s.\ no non-trivial left-orderable quotients. We proceed by constructing 3CNF formulas, which encode necessary conditions for left-orderability and then proving their unsatisfiability a.a.s.
\end{abstract}

\section{Introduction}

The~triangular binomial model of random groups is defined as follows. 

Fix $p:\mathbb{N}\rightarrow [0,1]$. Given $n\in\mathbb{N}$, let $S=\{s_1,s_2,\ldots,s_n\}$ be a set of~$n$~generators. A~random group in \emph{the~triangular binomial model $\Gamma(n,p)$} is given by the~presentation $\lle S | R\rre$, where $R$ is a~random subset of~the~set of~all cyclically reduced words of~length 3 over $S\cup S^{-1}$, with each word included in~$R$ independently with probability $p(n)$.

Given a~property $\mathcal{P}$ of~groups or presentations, we say that
a~random group in~the~model $\Gamma(n,p)$ satisfies $\mathcal{P}$ \emph{asymptotically almost surely} (\emph{a.a.s.}) if

\begin{equation*}
\lim_{n\rightarrow{\infty}}\mathbb{P}\big(\Gamma\in\Gamma(n,p)\text{ satisfies }\mathcal{P}\big) = 1.
\end{equation*}

This model (a variant of) was introduced by \.{Z}uk in~\cite{zuk03}. Basic properties of~random triangular groups vary with $p$ as described in the~following two theorems. Every bound on $p$ is assumed to hold for almost all $n$.

\begin{thm}[{\cite[Theorem 3]{zuk03}, \cite[Theorem 1]{als14}}]\label{phase_transition1}
	For any fixed $\eps>0$, if $p<n^{-3/2-\eps}$, then a~random group in $\Gamma(n,p)$ is a.a.s.\	infinite, torsion-free, and hyperbolic, while also there exists a~constant $C>0$ such that if $p> Cn^{-3/2}$, then a~random group in $\Gamma(n,p)$ is a.a.s.\ trivial.
\end{thm}

For sufficiently small values of~$p$, random triangular groups are actually free.

\begin{thm}[{\cite[Theorems 1, 2, 3]{als15}}]\label{phase_transition2}
	There exist constants $c_1,c_2>0$ such that a~random group in~$\Gamma(n,p)$ is a.a.s.\ free if $p<c_1n^{-2}$, and a.a.s.\ non-free or trivial if $p > c_2n^{-2}$.
\end{thm}

In~this article we explore left-orderability of~random triangular groups. A~particular consequence of~Theorem~\ref{phase_transition2} is that if $p<c_1n^{-2}$, then a~random group in~$\Gamma(n,p)$ is a.a.s.\ left-orderable (see \cite[Section 1.2.3]{dnr14}). We show that this statement is optimal up to multiplication of~$p$ by a~constant.

\begin{thmx}\label{main_thm_non_lo}
Let $c>(1/8)\log_{4/3}{2}\approx 0.3012$ be a~constant. If $p>cn^{-2}$, then a~random group in~$\Gamma(n,p)$ is a.a.s.\ trivial or not left-orderable. In particular, for any fixed $\varepsilon>0$, if $p\in(cn^{-2}, n^{-3/2-\eps})$, then a~random group in~$\Gamma(n,p)$ is a.a.s.\ not left-orderable.
\end{thmx}

Increasing $p$, we are able to show even more.

\begin{thmx}\label{main_thm_quot}
Suppose that $p\geq (1+\eps)(\log{n})n^{-2}$ for any fixed $\eps>0$. Then a~random group in~$\Gamma(n,p)$ has a.a.s.\ no non-trivial left-orderable quotients.
\end{thmx}

Theorem~\ref{main_thm_quot} is optimal up to a~constant factor. By \cite[Lemma 12]{als15}, if
$p\leq (1/25)(\log{n})n^{-2}$, then a.a.s.\ there exists a~generator $s\in S$ such that
neither $s$ nor $s^{-1}$ belongs to any relation in $R$, allowing for a~homomorphism
$\phi : \lle S | R \rre \twoheadrightarrow \mathbb{Z}$ with $\phi(s)=1$, and hence a~random group in~$\Gamma(n,p)$
has $\mathbb{Z}$ as its quotient a.a.s.

Note that if $p\geq n^{3d-3}$, for any fixed $d>1/3$, then $p$ satisfies the~hypothesis of Theorem~\ref{main_thm_quot} and hence its conclusion applies to random triangular groups in~the~\.{Z}uk's model at densities $d>1/3$ (see \cite[Section 3]{als15} for details). On the other hand, random groups in the~Gromov density model do not have non-trivial left-orderable quotients at any density $d\in (0,1)$, as shown in~\cite{orl17}.

A~countable group is left-orderable if and only if it admits a~faithful action on the real line by
orientation-preserving homeomorphisms (see \cite[Section 1.1.3]{dnr14}). Hence Theorem~\ref{main_thm_non_lo} (resp.~\ref{main_thm_quot}) is, equivalently, a~statement about non-existence of~faithful (resp. non-trivial) actions of~random triangular groups on the real line. Note that some constraints were previously known for actions of~random triangular groups on the~circle. Namely, by \cite[Theorem 3]{als15}, there exists a~constant $C'>0$ 
such that if $p\geq C'(\log{n})n^{-2}$, then a~random group $\Gamma\in\Gamma(n,p)$ has a.a.s.\ Kazhdan's property~(T), and hence by a~result of Navas, \cite{nav02}, every action of~$\Gamma$ on the~circle by orientation-preserving diffeomorphisms of~class $C^{1+\alpha}$ ($\alpha>1/2$) has a~finite image.

Another consequence of~our~proof of~Theorem~\ref{main_thm_non_lo} is a~new proof of~the~second part of~Theorem~\ref{phase_transition2} stating that a~random group in~$\Gamma(n,p)$ is a.a.s.\ non-free or trivial if $p>cn^{-2}$ for some constant ${c>0}$. However, our argument shows that this is the~case for~$c>(1/8)\log_{4/3}{2}\approx 0.3012$, which is not an optimal result. In~\cite{als15} this statement is formally shown to hold for $c\geq 3$, but the~method presented there works as soon as $c> 1/8=0.125$, as we briefly explain now. First, if $p\geq C'(\log n)n^{-2}$, then non-freeness or triviality follows from 
property (T). If $p$ is smaller, then in particular $p<n^{-3/2-\eps}$ for a~fixed $\eps>0$. This guarantees that the~natural presentation complex of~a~random $\Gamma\in\Gamma(n,p)$ is a.a.s.\ aspherical, and hence the~Euler characteristic of~$\Gamma$ is a.a.s.\ equal to ${\chi(\Gamma)=1-n+|R|}$. If in addition $p>cn^{-2}$ for a~fixed $c>1/8$, then a.a.s.\ $|R|=8n^3p(1+o(1))>n$, hence $\chi(\Gamma)>0$ and $\Gamma$ is not a~free group a.a.s.

\textbf{Outline of~the~proofs.}
In~the~proofs of~our results we use 3CNF propositional formulas, i.e. formulas of~form $(a_1 \vee b_1 \vee c_1)\wedge \ldots \wedge (a_k \vee b_k \vee c_k)$, where each $a_i, b_i, c_i$ is either of form $x$ or $\neg x$
for a~Boolean variable $x$.

To prove Theorem $\ref{main_thm_non_lo}$ for smaller values of $p$, we translate a~random triangular presentation $\Gamma=\lle S|R \rre$ into a~random 3CNF formula $\Phi_R$ constructed as the~conjunction of the~expressions of form 
\begin{equation*}
\big( x_i^{\eps_i} \vee x_j^{\eps_j} \vee x_k^{\eps_k}\big) \wedge \big( x_i^{-\eps_i} \vee x_j^{-\eps_j} \vee x_k^{-\eps_k}\big),
\end{equation*}
corresponding to relators $r=s_i^{\eps_i}s_j^{\eps_j}s_k^{\eps_k}\in R$, where $x_1, x_2, \ldots, x_n$ are Boolean variables and our~convention is that $x^1=x$ and $x^{-1}=\neg x$. Every left-order on $\Gamma$ a.a.s.\ induces a~truth assignment satisfying $\Phi_R$. We show, however, that a.a.s.\ $\Phi_R$ is unsatisfiable, hence $\Gamma$ is not left-orderable. For larger values of $p$ and in the proof of Theorem $\ref{main_thm_quot}$, we consider a~collection of~similarly constructed random 3CNF formulas $\Phi_{R,A}$, indexed by certain subsets $A\subseteq S$, such that a.a.s.\ every non-trivial left-orderable quotient of~a~random $\Gamma\in\Gamma(n,p)$ leads to satisfiability of at least one of~those~formulas. In this case we prove that a.a.s.\ none of the formulas $\Phi_{R,A}$ are satisfiable and the desired conclusion follows.

We note that the~constant $(1/8)\log_{4/3}{2}\approx 0.3012$ in~Theorem~\ref{main_thm_non_lo}
can probably be improved by using the existing research on the~random 3-NAESAT problem. To see the~connection, let $\Phi'_R$ be a~random formula constructed as the~conjunction of the clauses $x_i^{\eps_i} \vee x_j^{\eps_j} \vee x_k^{\eps_k}$, taken for each $r=s_i^{\eps_i}s_j^{\eps_j}s_k^{\eps_k}\in R$.
We say that $\Phi'_R$ is \emph{Not-All-Equal-satisfiable} (\emph{NAE-satisfiable}) if there exists a~truth assignment $\eta$ such that in every clause $x_i^{\eps_i} \vee x_j^{\eps_j} \vee x_k^{\eps_k}$ at least one, but not all, of~the~literals $x_i^{\eps_i}, x_j^{\eps_j}, x_k^{\eps_k}$ is true under $\eta$. It is straightforward to see that $\Phi_R$ is satisfiable if and only if $\Phi'_R$ is NAE-satisfiable. When $p=cn^{-2}$, a formula $\Phi'_R$ is similar to a~uniformly random 3CNF formula on $n$ variables with $\lfloor 8cn\rfloor$ clauses, which
is known to be a.a.s.\ not NAE-satisfiable if $8c$ is larger than a~certain threshold $\alpha$ (e.g. \cite{cop12}).
The conclusions of Theorem~\ref{main_thm_non_lo} should thus hold for any $c>\alpha/8$. In this article we do not aim to give the optimal constant for Theorem~\ref{main_thm_non_lo}, which allows us to provide a more self-contained proof.

\textbf{Organisation.} In Section~\ref{sec_prel} we introduce our notation, recall basic properties of left-orderable groups and construct propositional formulas central to our arguments. In Section~\ref{sec_non_lo} we prove Theorem~\ref{main_thm_non_lo} under the additional assumption that $p<n^{-5/3-\eps}$. The remaining case
is a~direct consequence of~Theorem~\ref{main_thm_quot}, which we prove in Section~\ref{sec_qout}.

\textbf{Acknowledgements.} The~author would like to thank Piotr Nowak and Piotr Przytycki for valuable discussions
and suggestions. The author was partially supported by the European Research Council (ERC) under the European Union's Horizon 2020 research and innovation programme (grant agreement no. 677120-INDEX) and by (Polish) Narodowe Centrum Nauki, UMO-2018/30/M/ST1/00668.

\section{Preliminaries}\label{sec_prel}

Most of~the notions we use depend implicitly on $n$. By $o(f(n))$ we denote any function $g(n)$
such that $g(n)/f(n)\rightarrow 0$ as $n\rightarrow\infty$. Throughout, we assume $\Gamma=\lle S | R\rre$ is group given by a~presentation with the~set of formal generators $S=\{s_1,s_2,\ldots, s_n\}$ and $R \subseteq \red$, where $W_3$ is the~set of all cyclically reduced words of length 3 over $S\cup S^{-1}$. We denote by $\iota : F(S) \twoheadrightarrow \Gamma$ the~associated epimorphism of the~free group generated by $S$. For any $A\subseteq S$ define

\begin{equation*}
R_A = \{s_i^{\eps_i}s_j^{\eps_j}s_k^{\eps_k} \in R : s_i, s_j, s_k \in A, \eps_i, \eps_j, \eps_k \in \{-1,1\} \}
\end{equation*}
to be the set of words over $A\cup A^{-1}$, belonging to $R$.

\begin{defin}\label{def_order}
We say that a~group $G$ is \emph{left-ordered} by a~linear order $\leq$ if $\leq$ is invariant under left-multiplication by $G$, i.e. satisfies the condition $a\leq b \implies ga\leq gb$ for all $a,b,g\in G$. We say that $G$ is \emph{left-orderable} if it is \emph{left-ordered} by some linear order $\leq$.
\end{defin}

Given a~group $G$ left-ordered by $\leq$, we use $<$ and $>$ as usual shorthands. The~following
properties of~left-orderable groups are straightforward consequences of~Definition~\ref{def_order}.

\begin{rem}\label{left_order_rem}
If a~group $G$ is left-ordered by $\leq$, then the~following hold.
\begin{enumerate}
	\item For every non-trivial $g\in G$, $g>1_G$ if and only if $g^{-1}<1_G$.
	\item If $g_1, g_2, \ldots, g_m \in G$ are such that $g_i>1_G$ for every $i$, then also $g_1g_2\ldots g_m > 1_G$.
\end{enumerate}
\end{rem}

Now we define the~propositional formulas central to our arguments. Let $x_1,x_2,\ldots, x_n$ be Boolean variables. We adopt a~convention that $x_i^1 = x_i$ and $x_i^{-1} = \neg x_i$ for all $i$. For every $r=s_i^{\eps_i} s_j^{\eps_j} s_k^{\eps_k} \in W_3$ with $\eps_i, \eps_j, \eps_k \in \{-1,1\}$, set

\begin{equation*}
\phi_r = \big(x_i^{\eps_i} \vee x_j^{\eps_j} \vee x_k^{\eps_k}\big)\wedge \big(x_i^{-\eps_i} \vee x_j^{-\eps_j} \vee x_k^{-\eps_k}\big).
\end{equation*} 

For any $A\subseteq S$, let

\begin{equation*}
\Phi_{R, A} = \bigwedge_{r\in R_A} \phi_r
\end{equation*}
and let $\Phi_R=\Phi_{R,S}$. The~definition of $\Phi_{R,A}$ is justified by the following.

\begin{prop}\label{formula_prop}
	Suppose $q : \Gamma \twoheadrightarrow Q$ is an~epimorphism onto a~left-orderable group $Q$, such that $\ker(q\iota)\cap A =\emptyset$.
	Then $\Phi_{R,A}$ is satisfiable.
\end{prop}

Of our special interest is the case when $Q=\Gamma$, $q=\text{id}_\Gamma$ and $A=S$.

\begin{cor}\label{single_formula}
	Suppose $\Gamma$ is left-orderable and $\ker(\iota) \cap S=\emptyset$. Then $\Phi_R$ is satisfiable.
\end{cor}

\begin{proof}[Proof of Proposition~\ref{formula_prop}]
	Let $Q$ be left-ordered by $\leq$.
	We construct a~truth assignment $\eta$, satisfying $\Phi_{R,A}$, as follows.
	Let $x_i$ be any variable such that $s_i\in A$. Then $q\iota(s_i)\neq 1_Q$. Set
	
	\begin{equation*}
		\eta(x_i) = \begin{cases}
			T & \text{if }q\iota(s_i)>1_Q,\\
			F & \text{if }q\iota(s_i)<1_Q.
		\end{cases}
	\end{equation*}	

	Function $\eta$ extends naturally to all propositional formulas over variables $x_i$, for which $s_i\in A$. By Remark~\ref{left_order_rem}(1), for every such variable $x_i$ and every $\eps\in\{-1,1\}$, the~value of~$\eta(x_i^\eps)$ represents the~validity of~the~statement $q\iota(s_i^\eps) > 1_Q$.
	
	Now consider any $r=s_i^{\eps_i} s_j^{\eps_j} s_k^{\eps_k} \in R_A$. Since $1_Q=q\iota(r)=q\iota\left(s_i^{\eps_i}\right)q\iota(s_j^{\eps_j})q\iota(s_k^{\eps_k})$, by Remark~\ref{left_order_rem}(2) at least one of the elements $q\iota\left(s_i^{\eps_i}\right)$, $q\iota(s_j^{\eps_j})$, $q\iota(s_k^{\eps_k})$ is not greater than $1_Q$ and hence $\eta$ satisfies the formula $x_i^{-\eps_i}\vee x_j^{-\eps_j} \vee x_k^{-\eps_k}$. Similarly, as 
	$1_Q=q\iota(r^{-1})=q\iota\left(s_k^{-\eps_k}\right)q\iota(s_j^{-\eps_j})q\iota(s_i^{-\eps_i})$, we establish
	that $\eta$ satisfies $x_i^{\eps_i}\vee x_j^{\eps_j} \vee x_k^{\eps_k}$ and so it satisfies $\phi_r$.
	
	Finally, $\eta$ satisfies $\Phi_{R,A}$, being a~conjunction of~formulas $\phi_r$ for $r\in R_A$.
	
\end{proof}

\section{Non-left-orderability for lower probabilities}\label{sec_non_lo}

In this section we prove Theorem~\ref{main_thm_non_lo} under the~additional assumption that
$p<n^{-5/3-\eps}$ for some $\eps>0$. The case of~larger $p$ follows from Theorem~\ref{main_thm_quot}, which is proved independently in~Section~\ref{sec_qout}.

Throughout, $\mathbb{P}$ is the~probability function in the~model $\Gamma(n,p)$.

Our choice of~an upper bound for $p$ is motivated by~the~following fact.

\begin{lem}[{\cite[Corollary 10]{als15}}]\label{non_triv_gen}
	Suppose $p<n^{-5/3-\eps}$ for a~constant $\eps>0$. Then a.a.s.\ ${\ker(\iota)\cap S = \emptyset}$
	in $\Gamma(n,p)$.
\end{lem}

Before proving Theorem~\ref{main_thm_non_lo} we need to establish two simple lemmas, which let us pass to a~different model of~randomness for~the~set of relators~$R$ in computations of probabilities. First we need some asymptotic control over the~size of~the~set~$R$.

\begin{lem}\label{lem_conc}
	Suppose $p>n^{-3+\eps}$ for some $\eps>0$. Then $|R|\in\left((1-\delta)8pn^3,(1+\delta)8pn^3 \right)$ a.a.s.\ for some $\delta=\delta(n)=o(1)$.
\end{lem}

\begin{proof}
	
	Under the~model $\Gamma(n,p)$, $|R|$ has the binomial distribution $B\left(|W_3|, p\right)$, so that
	$\mathbb{E}|R| = p|W_3|$
	 and $\Var{|R|} = p(1-p)|W_3|$. 
	Let $\delta = (p|W_3|)^{-1/3}$.
	
	By the~Chebyshev's inequality, 
	
	\begin{equation*}
	\mathbb{P}\big( \big||R|-\mathbb{E}|R| \big| \geq \delta\mathbb{E}|R|\big)\leq \frac{\Var |R|}{\delta^2(\mathbb{E}|R|)^2}
	= \frac{p(1-p)|W_3|}{\delta^2p^2|W_3|^2}\leq \frac{1}{\delta^2 p|W_3|}=\delta.
	\end{equation*}
	
	Now note that $|W_3|=8n^3(1+o(1))$, as $2n(2n-1)(2n-2)\leq |W_3| \leq (2n)^3$ and hence ${p|W_3| = 8pn^3(1+o(1))>8n^{\eps}(1+o(1))}$, so that $\delta = o(1)$.
	
	We have $|R|\in\left((1-\delta)\mathbb{E}|R|,(1+\delta)\mathbb{E}|R| \right)$ a.a.s. Since $\mathbb{E}|R| = 8pn^3(1+o(1))$, 
	$\delta$ can be adjusted so that the~desired conclusion holds.
	
\end{proof}

Let $c_0=(1/8)\log_{4/3}{2}$. From now on to the~end of the~section we assume that $p\in (cn^{-2}, n^{-5/3-\eps})$, where $\eps>0$ and $c>c_0$. Let $\delta$ be such that the~conditions of~Lemma~\ref{lem_conc} hold and denote $I_\delta =((1-\delta)8pn^3, (1+\delta)8pn^3)\cap \mathbb{N}$. 

Fix $m$ and let $r_1, r_2,\ldots, r_m$ be independent random uniform words picked from $W_3$. Let $\mathbb{P}_m$ be the~associated probability function and let $R_m=\{r_1,r_2,\ldots, r_m\}$. The~following lemma enables us to 
work with $R_m$ in place of $R$.

\begin{lem}\label{switch_prob}
	There exists $\theta=\theta(n)=o(1)$ such that, for every property $\mathcal{P}$ of subsets of~$W_3$, and~every $m\in I_\delta$,
	\begin{equation*}
	\mathbb{P}\big(R\text{ satisfies }\mathcal{P}\text{ }\big|\text{ }|R|=m \big) \leq (1+\theta)\ \mathbb{P}_m\big(R_m\text{ satisfies }\mathcal{P}\big).
	\end{equation*}
\end{lem}

\begin{proof}
	
	In $\Gamma(n,p)$, conditional on $|R|=m$, the~set $R$ is a~random uniform subset of $W_3$, of~size $m$.
	
	Let $\mathcal{D}_m$ be the~event that $r_1, r_2, \ldots, r_m$ are pairwise distinct. Conditional on $\mathcal{D}_m$, $R_m$ is also a~random uniform subset of~$W_3$, of size $m$. Hence
	
	\begin{equation*}
	\mathbb{P}\big(R\text{ satisfies }\mathcal{P}\text{ }\big|\text{ }|R|=m \big) = 
	\mathbb{P}_m\big(R_m\text{ satisfies }\mathcal{P}\text{ }\big|\text{ }\mathcal{D}_m\big)
	\leq \frac{\mathbb{P}_m\big(R_m\text{ satisfies }\mathcal{P}\big)}{\mathbb{P}_m\big(\mathcal{D}_m\big)}.
	\end{equation*}

	It suffices to show that $\mathbb{P}_m\big(\mathcal{D}_m\big)\geq 1-o(1)$ for every $m\in I_\delta$, with 
	the bound $o(1)$  depending only on~$n$.
	
	We first note that $m^2\leq 64(1+\delta)^2p^2n^6\leq64(1+\delta)^2n^{8/3-2\eps}=o(n^3)$. As $|W_3|=8n^3(1+o(1))$, we have $m^2/|W_3| \leq o(1)$. Finally, by the~Bernoulli's inequality	
	\begin{equation*}
	\begin{split}
	\mathbb{P}_m\big(\mathcal{D}_m\big) &= \left(1-\frac{1}{|W_3|}\right)\left(1-\frac{2}{|W_3|}\right)\ldots
	\left(1-\frac{m-1}{|W_3|}\right)\\ 
	&\geq \left(1-\frac{m-1}{|W_3|}\right)^{m-1}\geq 1-\frac{(m-1)^2}{|W_3|}\geq 1-\frac{m^2}{|W_3|}
	\geq 1-o(1),
	\end{split}
	\end{equation*}
	
	proving our claim.
	
\end{proof}

\begin{proof}[Proof of Theorem~\ref{main_thm_non_lo} for $p<n^{-5/3-\eps}$]
	Let $\Gamma\in\Gamma(n,p)$ be a~random triangular group. 
	Consider the~following event in $\Gamma(n,p)$.
	\begin{equation*}
	\mathcal{A}=\big\{\Gamma\text{ is left-orderable, }\ker(\iota)\cap S=\emptyset\text{ and }|R|\in I_\delta. \big\}
	\end{equation*}
	
	By~Lemmas~\ref{non_triv_gen} and~\ref{lem_conc}, it suffices to show that $\mathbb{P}\left(\mathcal{A}\right)\rightarrow 0$ as $n\rightarrow \infty$.
	
	If $\mathcal{A}$ holds, then by Corollary~\ref{single_formula} the~formula $\Phi_R$ is satisfiable. We bound $\mathbb{P}\left(\mathcal{A}\right)$ as follows, using Lemma~\ref{switch_prob}.
	
	\begin{equation}\label{bound_on_A}
	\begin{split}
	\mathbb{P}\big(\mathcal{A}\big) &\leq \sum_{m\in I_\delta} \mathbb{P}\big(\Phi_R\text{ is satisfiable and }|R|=m \big)\\
	&= \sum_{m\in I_\delta} \mathbb{P}\big(\Phi_R\text{ is satisfiable }\big|\text{ }|R|=m \big) \mathbb{P}\big(|R|=m\big)\\
	&\leq (1+\theta)\sum_{m\in I_\delta}\mathbb{P}_m\big(\Phi_{R_m}\text{ is satisfiable}\big) \mathbb{P}\big(|R|=m\big)\\
	&\leq (1+\theta)\max_{m\in I_\delta}\ \mathbb{P}_m\big(\Phi_{R_m}\text{ is satisfiable}\big)
	\end{split}
	\end{equation}
	
	Fix $m\in I_\delta$. Let $\mathcal{E}$ be the~set of~all $2^n$ truth assignments $\eta : 
	\{x_1, x_2,\ldots, x_n \} \rightarrow \{T, F\}$. Using the~independence of $r_1,r_2,\ldots, r_m$,
	we obtain
	
	\begin{equation}\label{single_P_m}
	\begin{split}
	\mathbb{P}_m\big(\Phi_{R_m}\text{ is satisfiable}\big) &\leq \sum_{\eta\in\mathcal{E}}
	\mathbb{P}_m\big(\eta\text{ satisfies }\Phi_{R_m}\big)
	=\sum_{\eta\in\mathcal{E}}\mathbb{P}_m\left(\eta\text{ satisfies }\bigwedge_{i=1}^m\phi_{r_i}\right)\\
	&= \sum_{\eta\in\mathcal{E}}\mathbb{P}_m\big(\eta\text{ satisfies }\phi_{r_i}\text{ for }i=1,2,\ldots,m\big)\\
	&= \sum_{\eta\in\mathcal{E}}\ \prod_{i=1}^m\mathbb{P}_m\big(\eta\text{ satisfies }\phi_{r_i}\big)
	=\sum_{\eta\in\mathcal{E}}\ \big(\mathbb{P}_m\big(\eta\text{ satisfies }\phi_{r_1}\big)\big)^m.
	\end{split}
	\end{equation}
	
	Now fix $\eta$. For any 3 pairwise distinct $s_i, s_j, s_k \in S$, there exist exactly 2 triples
	${(\eps_i, \eps_j, \eps_k)\in\{-1,1\}^3}$, and hence exactly 2 words of form
	$w=s_i^{\eps_i}s_j^{\eps_j}s_k^{\eps_k}$, such that $\eta$ does not satisfy
	
	\begin{equation*}
	\phi_w = \big(x_i^{\eps_i} \vee x_j^{\eps_j} \vee x_k^{\eps_k}\big)\wedge \big(x_i^{-\eps_i} \vee x_j^{-\eps_j} \vee x_k^{-\eps_k}\big).
	\end{equation*}
	
	Hence in total there are at least $2n(n-1)(n-2)=2n^3(1+o(1))$ such words $w\in W_3$ that
	$\eta$~does not satisfy $\phi_w$, so
	\begin{equation*}
	\mathbb{P}_m\left(\eta\text{ satisfies } \phi_{r_1}\right) \leq 1-\frac{2n^3}{|W_3|}(1+o(1))
	= 1-\frac{2n^3}{8n^3}(1+o(1))=\frac{3}{4}+o(1),
	\end{equation*}
	where the bound $o(1)$ depends only on $n$. From (\ref{bound_on_A}) and (\ref{single_P_m}) this leads to

	\begin{equation*}
	\begin{split}
	\mathbb{P}\left(\mathcal{A}\right)&\leq(1+\theta)\max_{m\in I_\delta}\ \mathbb{P}_m\big(\Phi_{R_m}\text{ is satisfiable}\big)\\ 
	&\leq(1+\theta)\max_{m\in I_\delta}\ 2^n \left(\frac{3}{4}+o(1)\right)^m\\
	&\leq (1+\theta)\ 2^n \left(\frac{3}{4}+o(1)\right)^{(1-\delta)8pn^3}\\
	&=(1+\theta)\left(2\left(\frac{3}{4}+o(1)\right)^{(1-\delta)8pn^2}\right)^n.
	\end{split}
	\end{equation*}
	
	As $p>cn^{-2}$ with $c>c_0$, we have
	
	\begin{equation*}
	\begin{split}
	\mathbb{P}\left(\mathcal{A}\right)\leq(1+\theta)\ \left(2\left(\frac{3}{4}+o(1)\right)^{(1-\delta)8c}\right)^n.
	\end{split}
	\end{equation*}
	
	Finally, note that the~base of the~power tends to
	
	\begin{equation*}
 	2\left(\frac{3}{4}\right)^{8c} < 2\left(\frac{3}{4}\right)^{8c_0}=2\left(\frac{3}{4}\right)^{\log_{4/3}2}=1,
	\end{equation*}
	
	as $n\rightarrow \infty$, hence $\mathbb{P}\left(\mathcal{A}\right) \rightarrow 0$.
	
\end{proof}

\section{Non-left-orderability of quotients}\label{sec_qout}

In this section we prove Theorem~\ref{main_thm_quot}. The condition that $p$ is slightly larger than $(\log n)n^{-2}$ is dictated by the~following lemma,
which ensures that a~positive proportion of~elements of~$S$ represent non-trivial elements after passing to a~non-trivial quotient $Q$.

\begin{lem}\label{non_vanishing_lem}
	Suppose that $p\geq (1+\eps)(\log{n})n^{-2}$ for some $\eps>0$. Then there exists 
	a~constant $\alpha \in (0,1)$ such that a~random group $\Gamma \in \Gamma(n,p)$ a.a.s.\ satisfies the following property:
\begin{equation}\tag{$\star$}
\text{For every non-trivial epimorphism }q:\Gamma \twoheadrightarrow Q\text{, }\left| \ker(q\iota) \cap S\right| < \alpha n.
\end{equation}
\end{lem}

\begin{proof}
Choose $\alpha\in(0,1)$ so that $\alpha^2(1+\eps)>1$. 
For every subset $A\subseteq S$, define in $\Gamma(n,p)$ the~event 
\begin{equation*}
V_A = \left\{\text{There exists a non-trivial epimorphism }q:\Gamma \twoheadrightarrow Q
\text{ such that }\ker(q\iota)\cap S= A. \right\}
\end{equation*}
and the set of words
\begin{equation*}
P_A = \left\{abc : a, b\in A, c\in S\setminus{A} \right\} \subseteq W_3.
\end{equation*} 

We claim that $V_A\subseteq \{ P_A\cap R = \emptyset\}$.
Indeed, suppose that $V_A$ holds and $w=abc \in P_A$. Then $q\iota(w) = q\iota(a)q\iota(b)q\iota(c)=q\iota(c)\neq 1_Q$ 
and hence $w\notin R$. 

Note that if $|A|\geq \alpha n$, then $|P_A| = |A| \cdot |A| \cdot (n-|A|)\geq \alpha^2n^2(n-|A|)$, so that

\begin{equation*}
\mathbb{P}\big(V_A\big) \leq \mathbb{P}\big(P_A\cap R = \emptyset\big)=\left(1-p\right)^{|P_A|}
\leq \left(1-p\right)^{\alpha^2n^2(n-|A|)}
\leq e^{-p\alpha^2n^2(n-|A|)},
\end{equation*}
where we use the~bound $e^x\geq 1+x$, true for $x\in\mathbb{R}$. Also note that $\mathbb{P}\big(V_S\big)=0$.

Now we can bound the~probability that the property ($\star$) does not hold as follows.

\begin{equation*}
\setlength{\jot}{2ex}
\begin{split}
\mathbb{P}\big(\Gamma \in \Gamma(n,p)\text{ does not satisfy property }(\star)\big)
&\leq \sum\limits_{\substack{A\subseteq S:\\ \alpha n \leq |A| \leq n-1}}\mathbb{P}\big(V_A\big)
\leq\sum\limits_{\substack{A\subseteq S:\\ \alpha n \leq |A| \leq n-1}} e^{-p\alpha^2n^2(n-|A|)}\\
&\leq\sum\limits_{\substack{A\subseteq S:\\ |A| \leq n-1}} e^{-p\alpha^2n^2(n-|A|)}
=\sum\limits_{k=0}^{n-1} {n\choose k}e^{-p\alpha^2n^2(n-k)}\\
&=\sum\limits_{k=0}^{n-1} {n\choose k}\Big(e^{-p\alpha^2n^2}\Big)^{n-k}=\Big(1+e^{-p\alpha^2n^2}\Big)^n-1\\
&\leq \Big(e^{e^{-p\alpha^2n^2}}\Big)^n-1=e^{e^{\log{n}-p\alpha^2n^2}}-1.
\end{split}
\end{equation*}

To finish the proof of the lemma, it suffices to show that the last expression tends to 0 as~$n\rightarrow \infty$. This is true as the topmost exponent satisfies

\begin{equation*}
\begin{split}
\limsup_{n\rightarrow\infty}\left(\log{n}-p\alpha^2n^2\right) &\leq 
\limsup_{n\rightarrow\infty}\left(\log{n}-(1+\eps)(\log{n})n^{-2}\alpha^2n^2\right)\\
&=\limsup_{n\rightarrow\infty}\left(1-\alpha^2(1+\eps)\right)\log{n}\\
&=-\infty.
\end{split}
\end{equation*}

\end{proof}

\begin{proof}[Proof of Theorem~\ref{main_thm_quot}]
	Let $\Gamma \in \Gamma(n,p)$ be a~random triangular group and let $\alpha\in(0,1)$ be a~number satisfying Lemma~\ref{non_vanishing_lem}. Define in $\Gamma(n,p)$ the event
	\begin{equation*}
	\mathcal{S} = \left\{\text{There exists a non-trivial epimorphism }q:\Gamma\twoheadrightarrow Q\text{ with a left-orderable }Q\text{ and }\left| \ker(q\iota) \cap S\right| < \alpha n\right  \}.
	\end{equation*}

It suffices to prove that $\mathbb{P}\left(\mathcal{S}\right)\rightarrow 0$ as $n\rightarrow\infty$.
Suppose $\mathcal{S}$ holds and let $A=S\setminus \ker(q\iota)$. We have $|A|\geq (1-\alpha)n$ and by~Proposition~\ref{formula_prop} the formula $\Phi_{R,A}$ is satisfiable. 

Let $\eta : \{x_i : s_i \in A\}\rightarrow \{T, F\}$ be a truth assignment satisfying $\Phi_{R,A}$. For every $s_i\in A$, choose $\eps_i \in \{-1, 1\}$ so that
$\eta(x_i^{\eps_i})=T$. Introduce the set 

\begin{equation*}
P_{A,\eta} = \{ s_i^{\eps_i}s_j^{\eps_j}s_k^{\eps_k} : s_i, s_j, s_k \in A \},
\end{equation*}
of cyclically reduced words, of cardinality $|P_{A,\eta}|\geq (1-\alpha)^3n^3$. By design, if $w\in P_{A, \eta}$, then $\eta$~does not satisfy $\phi_w$, so that $w\notin R_A$ and hence $P_{A,\eta}\cap R=P_{A,\eta}\cap R_A=\emptyset$. 

For fixed $A$ and $\eta$ we have

\begin{equation*}
\mathbb{P}\left(P_{A,\eta}\cap R=\emptyset\right) \leq (1-p)^{(1-\alpha)^3n^3}.
\end{equation*}

Denoting by $\mathcal{E}_A$ the~set of~all truth assignments 
$\eta : \{x_i : s_i \in A\}\rightarrow \{T, F\}$, we can bound

\begin{equation*}
\begin{split}
\mathbb{P}\left(S\right) &\leq \sum\limits_{\substack{A\subseteq S:\\ (1-\alpha) n \leq |A|}}
\sum\limits_{\eta\in\mathcal{E}_A} \mathbb{P}\left(P_{A,\eta}\cap R=\emptyset\right) \\
&\leq 2^n 2^n (1-p)^{(1-\alpha)^3n^3} \\
&\leq 4^n e^{-p(1-\alpha)^3n^3}\\
&=e^{n\log{4}-p(1-\alpha)^3n^3}.
\end{split}
\end{equation*}

As $p\geq (\log n)n^{-2}$, we have 

\begin{equation*}
e^{n\log{4}-p(1-\alpha)^3n^3} \leq e^{n\log{4} -(1-\alpha)^3n\log{n}}=e^{n(\log{4} -(1-\alpha)^3\log{n})}\rightarrow 0
\end{equation*}

as $n\rightarrow\infty$ and hence $\mathbb{P}\left(\mathcal{S}\right)\rightarrow 0$ as $n\rightarrow\infty$.

\end{proof}

\end{document}